\documentclass[english]{lematema}

\title{A short note on Cayley-Salmon equations}
\titlemark{A short note on Cayley-Salmon equations}


\usepackage{graphicx}

\author{Marvin Anas Hahn}
\address{%
 Institut f\"ur Mathematik\\
Goethe-Universit\"at Frankfurt,\\
\email{hahn@math.uni-frankfurt.de}
}
\author{Sara Lamboglia}
\address{%
 Institut f\"ur Mathematik\\
Goethe-Universit\"at Frankfurt,\\
\email{lamboglia@math.uni-frankfurt.de}
}
\author{Alejandro Vargas}
\address{%
Mathematisches Institut\\
Universit\"at Bern\\
\email{a.vargas.math@gmail.com}
}

\date{2019/09/13}
\newcommand{\inv}[1]{{#1}^{-1}}
\newcommand{\cone}[2]{\text{Cone}_{#1} (#2)}

\newcommand{\cayley}{Cayley-Salmon equation}
\newcommand{\cayleys}{Cayley-Salmon equations}
\begin{document}

\maketitle

\begin{abstract}
A \cayley~for a smooth cubic surface $S$ in $\mathbb P^3$ is an expression of the form $l_1l_2l_3 - m_1m_2m_3 = 0$ such that the zero set is~$S$ and $l_i$, $m_j$ are homogeneous linear forms. This expression was first used by Cayley and  Salmon to study the incidence relations of the 27 lines on~$S$. There are 120 essentially distinct \cayleys~for $S$.  In this note we  give an exposition of a classical proof of this fact. We illustrate the explicit calculation to obtain these equations and we apply it to the Clebsch surface and to the octanomial form in \cite{pss19}. Finally we show that these $120$  \cayleys~ can be directly computed  using recent work by Cueto and Deopurkar  \cite{C-D}.

	
\end{abstract}

\section{Introduction and motivation}

During the early days of algebraic geometry results were mostly derived via clever constructions in projective geometry and lengthy manipulation of the equations which defined the varieties. In the latter case it was important to choose a convenient \emph{normal form} for the equations. For example, it is immediate to deduce the presence of lines on a surface $S\subset \mathbb P^3$ defined by a polynomial of the form $l_1l_2l_3 - m_1m_2m_3$, with $l_i$ and $m_j$ linear forms. In fact the lines are   the intersection  of the planes $\{l_i = 0\} \cap \{m_j = 0\}$. This was first noticed and exploited by Cayley and Salmon to investigate the incidence relations of the 27 lines on a smooth cubic surface. 

We call an equation of the form \begin{align} \label{def:steinerform} l_1l_2l_3 - m_1m_2m_3 = 0,\end{align} with $l_i$ and $m_j$ homogeneous linear forms with four variables, a \emph{\cayley}. 
The first half of this note is an exposition on a classical proof of the fact that a smooth cubic surface $S$ in $\mathbb P^3$ can be expressed as the zero-set of a \cayley~in 120 essentially distinct ways. The arguments we reproduce illustrate the kind of geometric constructions  used by Cayley, Salmon and later by Steiner. On the way we show that $S$ has 27 lines, 45 tritangent planes that intersect $S$ at 3 distinct lines, and 120 pairs of triples of tritangent planes where each triple of planes altogether intersects $S$ in nine lines.  The conclusion of the classical constructions is that knowing the equations for the 27 lines we can compute the \cayleys, and conversely, from a \cayley~it is possible to derive equations for the 27 lines. In Section \ref{sec: computing} we compute the  \cayleys{} first for the Clebsch surface and then for the   octanomial model in \cite{pss19}. The computations are implemented in \verb|Macaulay2| \cite{M2} and \verb|Sage| \cite{sage} and are available at the supplementary website\footnote{\texttt{https://github.com/Saralamboglia/CayleySalmonEquations}}.

%
 %
 
In the second part of the note we treat the question of calculating the \cayleys~without having prior knowledge of the 27 lines. We first introduce the more modern approach of considering $S$ as a del Pezzo surface of degree $3$. \cite[section V.4]{hartshorne2013algebraic}
This is a  complete non-singular surface with ample anticanonical bundle with self intersection number $3$. It is well-known that degree 3 del Pezzo surfaces may be realized as smooth cubic surfaces in $\mathbb{P}^3$ 
%
and  as the blow-up of $\mathbb P^2$ at six generic points $p_1,\ldots,p_6$, i.e. no three of them are colinear and no conic passes though all six of them (see \cite[Section I.4 and Section 5 Proposition 4.10 ]{hartshorne2013algebraic}. Furthermore, in  a family of cubics containing six generic points  there are at least $24$ cuspidal cubics (it follows from  \cite{Bra}, for a more modern explanation  \cite[Section 4]{rsss13}). Hence one can assume  that after coordinate change the points $p_1,\ldots,p_6$  lie on the cuspidal cubic $ V(z^3-xy^2)$  and therefore are of the form
\begin{equation}
\label{equ:points}
\begin{pmatrix}
1 & 1 & 1 & 1 & 1 & 1\\
d_1 & d_2 & d_3 & d_4 & d_5 & d_6\\
d_1^3 & d_2^3 & d_3^3 & d_4^3 & d_5^3 & d_6^3
\end{pmatrix}.
\end{equation}

 One of the main results of \cite{C-D} gives an explicit description of the ideal of an embedding of the surface $S$ into $\mathbb{P}^{44}$. More precisely, they compute the ideal of the image of this embedding in terms of $d_1,\dots,d_6$. We compute   the pull-back of this ideal  via a suitable morphism  $
\mathbb{P}^3 \rightarrow\mathbb{P}^{44}$ recovering all \cayleys    \;(see Proposition \ref{prop:sec 4}).

In recent years, the tropicalization of cubic surface has gotten increased attention. Especially the fact that tropicalizations of cubic surfaces may have more than $27$ lines -- they might have infinitely many lines -- has been a driving force for several new developments. The works \cite{C-D} and \cite{pss19} both aim for a better understanding of this feature, although from different perspectives. Ultimately our exposition, examples and calculations are intended to be basic tools for exploring possible implications of \cayleys~to a deeper understanding of the tropicalizations of cubic surfaces. The 120 Cayley-Salmon equations of a cubic surface contain the information of its $27$ lines. 
Therefore, a possible next step in exploring such implications is a study of \textit{tropical Cayley-Salmon equations} in the flavour of the classical approach of Cayley and Salmon. 


\section{Classical constructions of \cayleys}\label{section:classical}
Consider  $X,S \subset \mathbb P^3$ to be smooth surfaces, with  $S$ of degree $3$. The history of the 27 lines goes back to the 19th century, when Arthur Cayley observed via a parameter count that a generic cubic surface can contain only finitely many lines. He corresponded with Salmon, who replied with an algebro-combinatorial argument showing that if a smooth cubic surface contains one line $L$ then it contains exactly 27 lines. See the discussion at the end of \cite{cay49}. We follow the exposition in \cite{salmon1} and \cite{salmon2} to: sketch how to prove that $S$ has at least one line $L$, prove there are 27 lines and 45 tritangent planes associated to~$S$, and show that $S$ can be expressed as a \cayley~in 120  distinct ways.
 
  
\subsection{Cayley-Salmon's proof for the existence of a line in $S$}
\label{section:classical-way}
Here we summarize Cayley's proof in  \cite{cay49} of the following fact:

\begin{thm} \label{thm:one-line}
	Let $S \subset \mathbb P^3$ be a smooth cubic surface. Then there exists a line $L$ contained in $S$.
\end{thm}

The proof goes through several geometrical facts known at the time about tangent spaces, a clever geometrical construction, and the Pl\"ucker formulas for plane curves. We do not give proofs to the lemmas, these follow mostly from straightforward calculations. Instead we focus on discussing their consequences. When a deeper insight is needed for a proof, we sketch the idea.

The first objects to consider are  tangent planes to surfaces. Denote by $T_P X$ the tangent plane of an algebraic surface $X$ at a point~$P = [p_0:p_1:p_2:p_3]\in X$. If $X = V(f)$ then $T_P X$ is the zero set of $x_0 \frac \partial {\partial x_0} f|_P + x_1 \frac \partial {\partial x_1} f|_P + x_2 \frac \partial {\partial x_2} f|_P+ x_3 \frac \partial {\partial x_3} f|_P$.

\begin{lemma} \label{lemma:singular-point}
	Let $X \subset \mathbb P^3$ be a smooth surface, $\Pi \subset \mathbb P^3$ a plane such that $\Pi \not \subset X$,  $C = X \cap \Pi$ the scheme theoretic intersection, and $P \in C$. Then $\Pi = T_P X$ if and only if $C$ is singular at $P$. 
\end{lemma}

\begin{lemma} \label{lemma:union-three-lines}
	Let $S$ be a smooth cubic surface, $P \in S$ a point, and $C = S \cap T_P S$ a curve. Then $C$ is either a singular irreducible cubic, the union of a line and a conic, or the union of three distinct lines. 
\end{lemma}

We are interested in studying the tangent lines to a surface $X$ passing through some point $A$ outside $X$. For that purpose we introduce the following object:

\begin{dfn}
	Given a point $A = [a_0 : a_1 : a_2 : a_3] \in \mathbb P^3$ and $X = V(f) \subset \mathbb P^3$   a surface, the \emph{first polar} with respect to $A$ is the surface in $\mathbb P^3$ defined as the zero set of the polynomial:
	\begin{align}
		f_A := a_0 \frac {\partial f} {\partial x_0}  + a_1 \frac {\partial f} {\partial x_1} + a_2 \frac {\partial f} {\partial x_2} + a_3 \frac {\partial f} {\partial x_3}.
	\end{align}
\end{dfn}

Geometrically, we have that $P  \in V(f) \cap V(f_A)$ if and only if the line $\overline{PA}$ is tangent to $V(f)$ at $P$. For $\deg f = 3$ the first polar appears naturally in the following Taylor expansion, where $[\lambda: \mu] \in \mathbb P^1$ parametrizes the line $\overline{PA}$:
\begin{align} \label{eq:taylor}
	f(\lambda P + \mu A) = \lambda^3 f(P) + \lambda^2 \mu \cdot f_A(P) + \lambda \mu^2 f_P(A) + \mu^3 f(A).
\end{align}
If $\overline{PA}$ is tangent to $S = V(f)$ at $P$ and at $A$, then the right hand side of Equation~\eqref{eq:taylor} is identically zero, namely $f(\lambda P + \mu A) = 0$, so $\overline{PA} \subset S$. This gives:

\begin{lemma} \label{lemma:touch-two-points}
	Let $S$ be a smooth cubic surface. If a line $L$ is tangent to $S$ in two distinct points then $L \subset S$.
\end{lemma}

If $X = V(f)$, we denote by $X_A$ the first polar $V(f_A)$. The main construction to prove the existence of lines on $S$
is the cone over the algebraic curves $S\cap S_A$ with vertex $A$.
\begin{dfn}
	The \emph{tangent cone} to $X$ with vertex $A$ is the set of lines passing through the point $A$ and tangent to $X$, namely 
	\begin{align}
		\cone A X := \bigcup_{P \in X \cap X_A} \overline{PA}.
	\end{align}
\end{dfn}

\begin{lemma} \label{lemma:degree-cone}
	Let $X = V(f)$, $n = \deg f$, and $A \in \mathbb P^3$ a point not in $X$. Then $\cone A X$ is an algebraic surface of degree $n(n-1)$.
\end{lemma}

	To see geometrically that the degree of $\cone A X$ is $n(n-1)$, let $\Pi$ be a general plane through $A$. The plane curves $X \cap \Pi$ and $X_A \cap \Pi$ have degrees $n$ and $(n-1)$ respectively, so by B\'ezout the set $X \cap X_A \cap \Pi$ has $n(n-1)$ points. Hence, the curve $\cone A X \cap \Pi$ consists of the $n(n-1)$ lines containing $A$ and the points of $X \cap X_A \cap \Pi$, so it has degree $n(n-1)$, giving the degree of $\cone A X$.

Now comes the ingenious idea of Cayley and Salmon for finding lines in a surface $S$ of degree~3. Let $A \not \in S$ and $\tilde \Pi$ be a plane with $A \not \in \tilde \Pi$. The idea is to set up a chain of bijections between the lines on $S$ and lines in $\tilde \Pi$ that are tangent to the curve $\cone A S \cap \tilde \Pi$ at two points. These lines are called \emph{bitangent} lines and can be counted  using  \emph{Pl\"ucker formulas}. The chain of bijections is: 
\[\left\{\parbox[c]{1.5cm} {\centering lines in $S$} \right\}
\longleftrightarrow
\left\{\parbox[c]{2cm} {\centering \small $\Pi$ tangent to~$S$ at two points in $S \cap S_A$} \right\}
 \longleftrightarrow
\left\{\parbox[c]{1.8cm} {\centering \small $\Pi$ tangent to~$\cone A S$ at two points in $S \cap S_A$ } \right\}
\longleftrightarrow
\left\{\parbox[c]{2cm} {\centering \small bitangents to $\cone A S \cap \tilde \Pi$} \right\}.
\]

\begin{exa} \label{example:clebsch}
	We exemplify the bijection between lines in $S$ and bitangents to $\cone A S \cap \tilde \Pi$ with a remarkable smooth cubic surface whose 27 lines are real, the \emph{Clebsch} surface. We take the Clebsch to be given by the polynomial:
	\begin{align*} f &= (x+\sqrt 3y+z/4)^3 + (x - \sqrt 3y + z/4)^3 \\&+ (-2x + z/4)^3 + (3z/4 + w)^3 - (3z/2 + w)^3.
	\end{align*}
	For $A$ and $\tilde \Pi$ we take the point $(0,0,-1/2,1)$ and the plane $x_2 = 0$, respectively.
	
	\vspace{1em}
	
	\noindent
	\begin{minipage}{0.35\textwidth}
		\centering
		\includegraphics[scale=0.12]{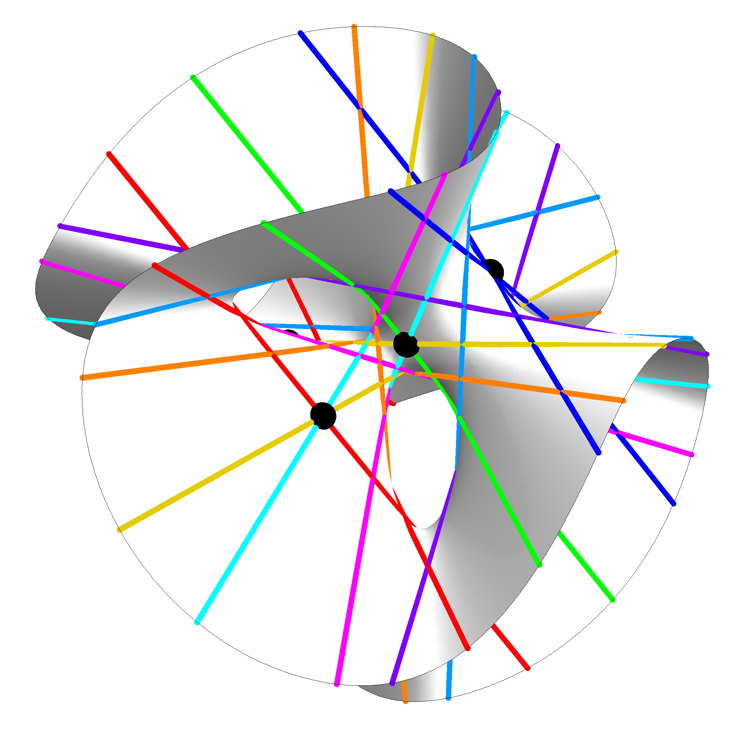}
		
	\end{minipage}
	\begin{minipage}{0.3\textwidth}
		\centering
		\includegraphics[scale=0.3]{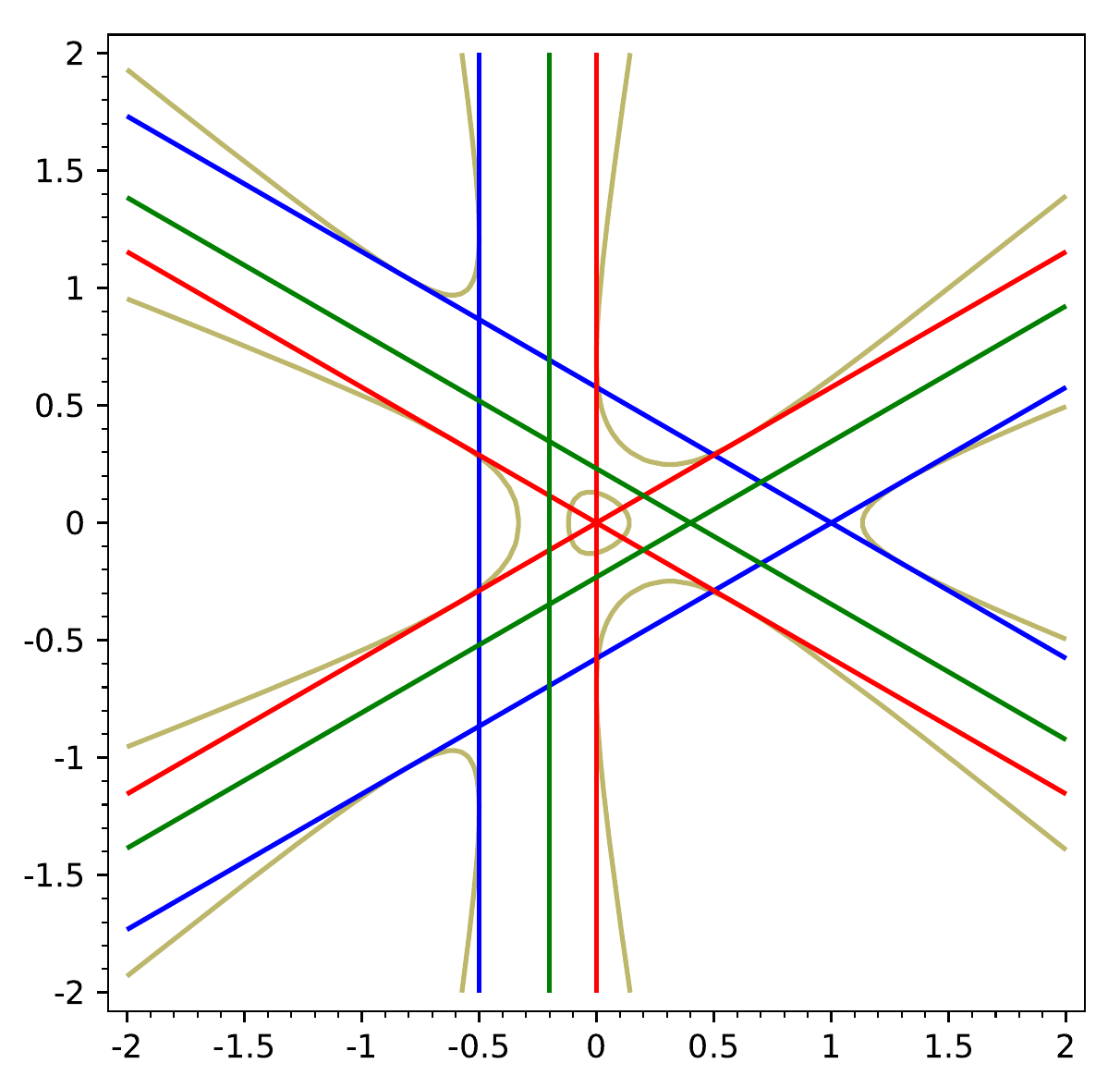}
	\end{minipage}
	\begin{minipage}{0.3\textwidth}
		\centering
		\includegraphics[scale=0.3]{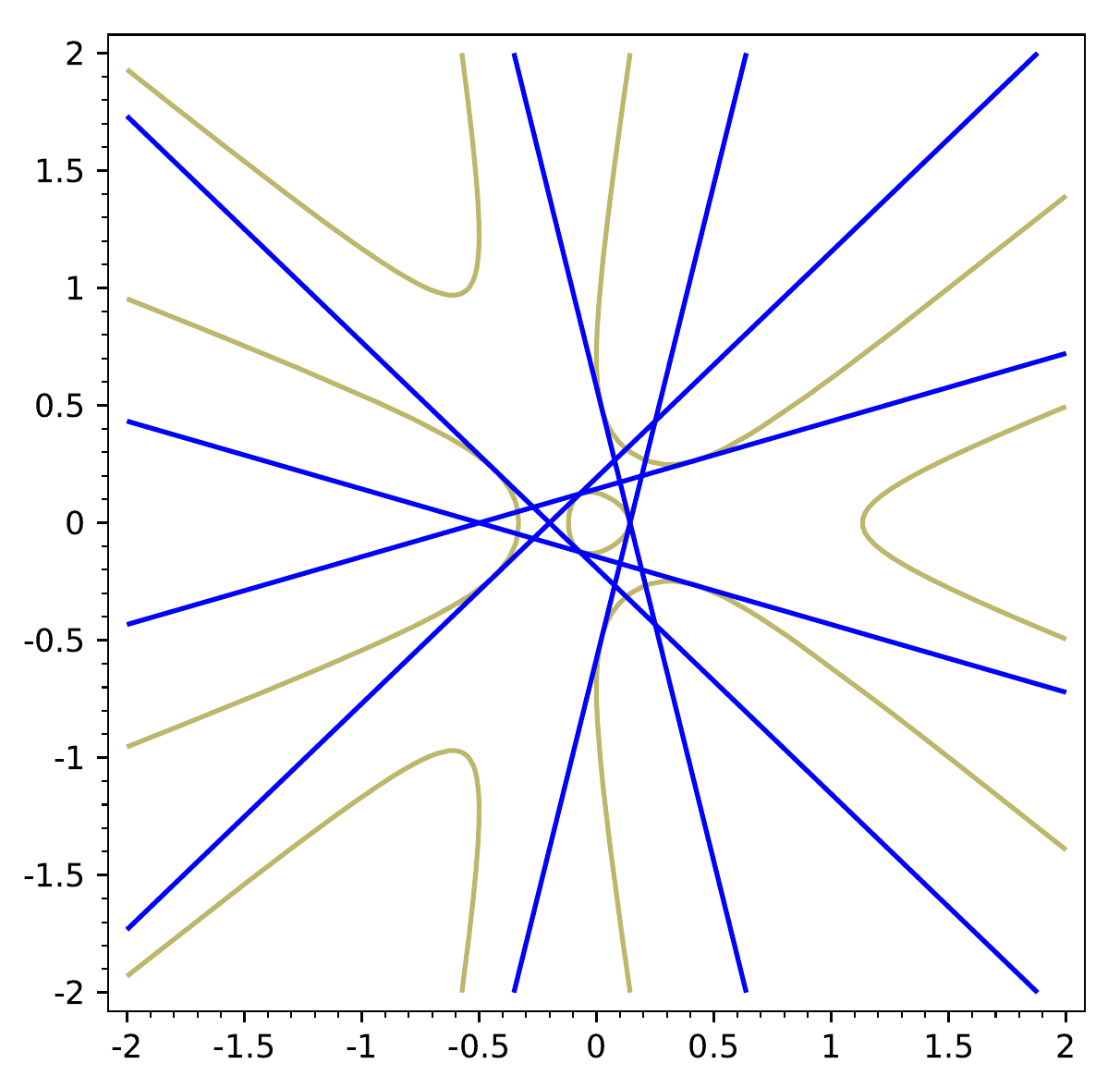}
	\end{minipage}

	\noindent
\begin{minipage}{0.35\textwidth}
	\centering
	\noindent \scriptsize Image by Greg Egan, see \cite{greg}.
\end{minipage}
\begin{minipage}{0.3\textwidth}
	\centering
\end{minipage}
\begin{minipage}{0.3\textwidth}
	\centering
\end{minipage}
		\vspace{0.7em}

	\noindent
\begin{minipage}{0.35\textwidth}
	\centering
	
	\includegraphics[scale=0.2]{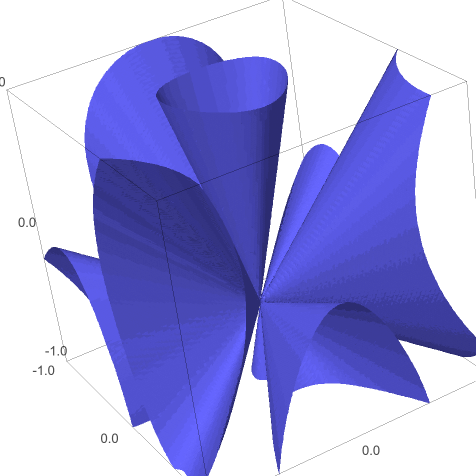}
	\vspace{0.3em}

\end{minipage}
\begin{minipage}{0.3\textwidth}
	\centering

	\includegraphics[scale=0.3]{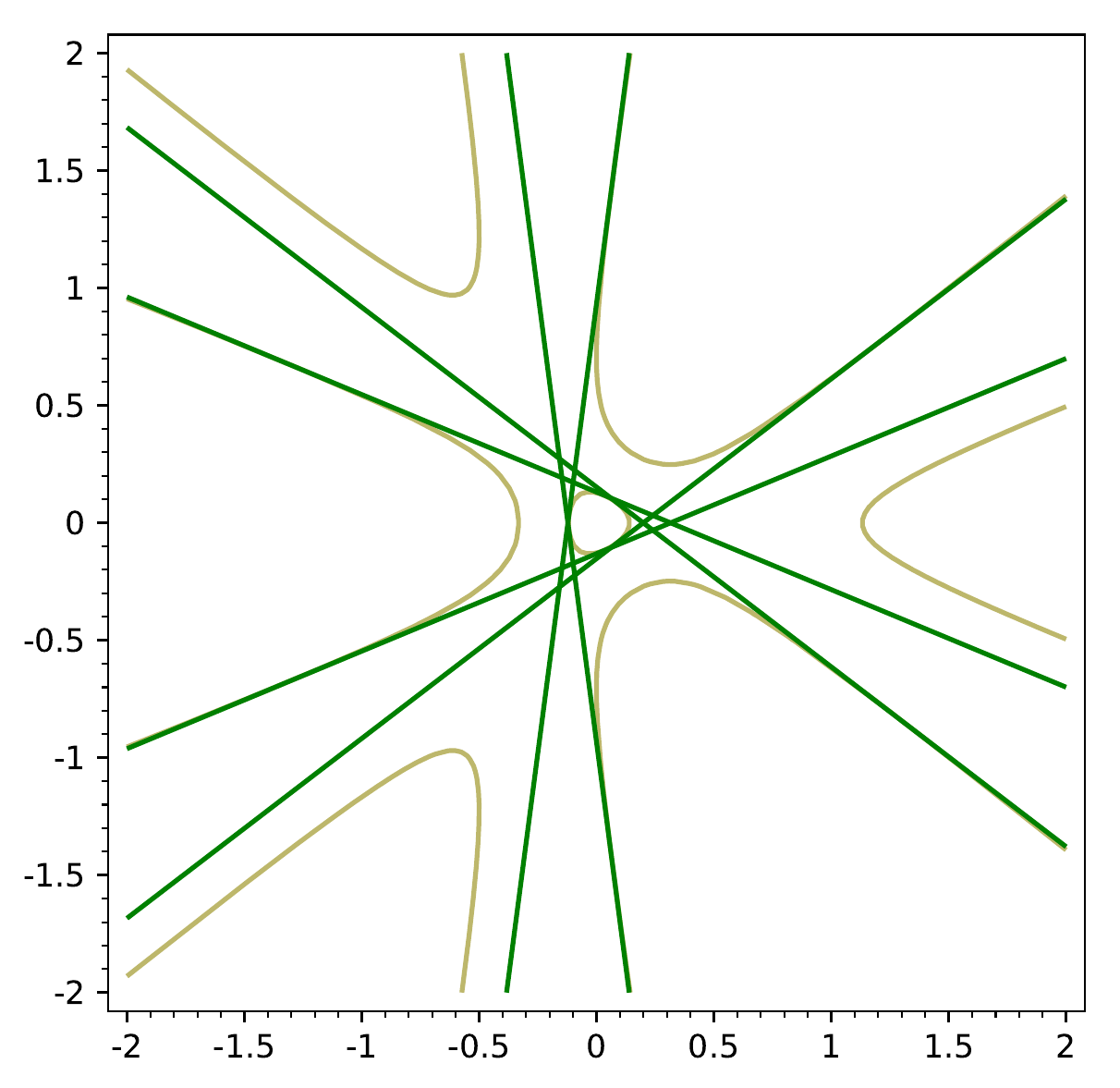}
\end{minipage}
\begin{minipage}{0.3\textwidth}
	\centering	
	\includegraphics[scale=0.3]{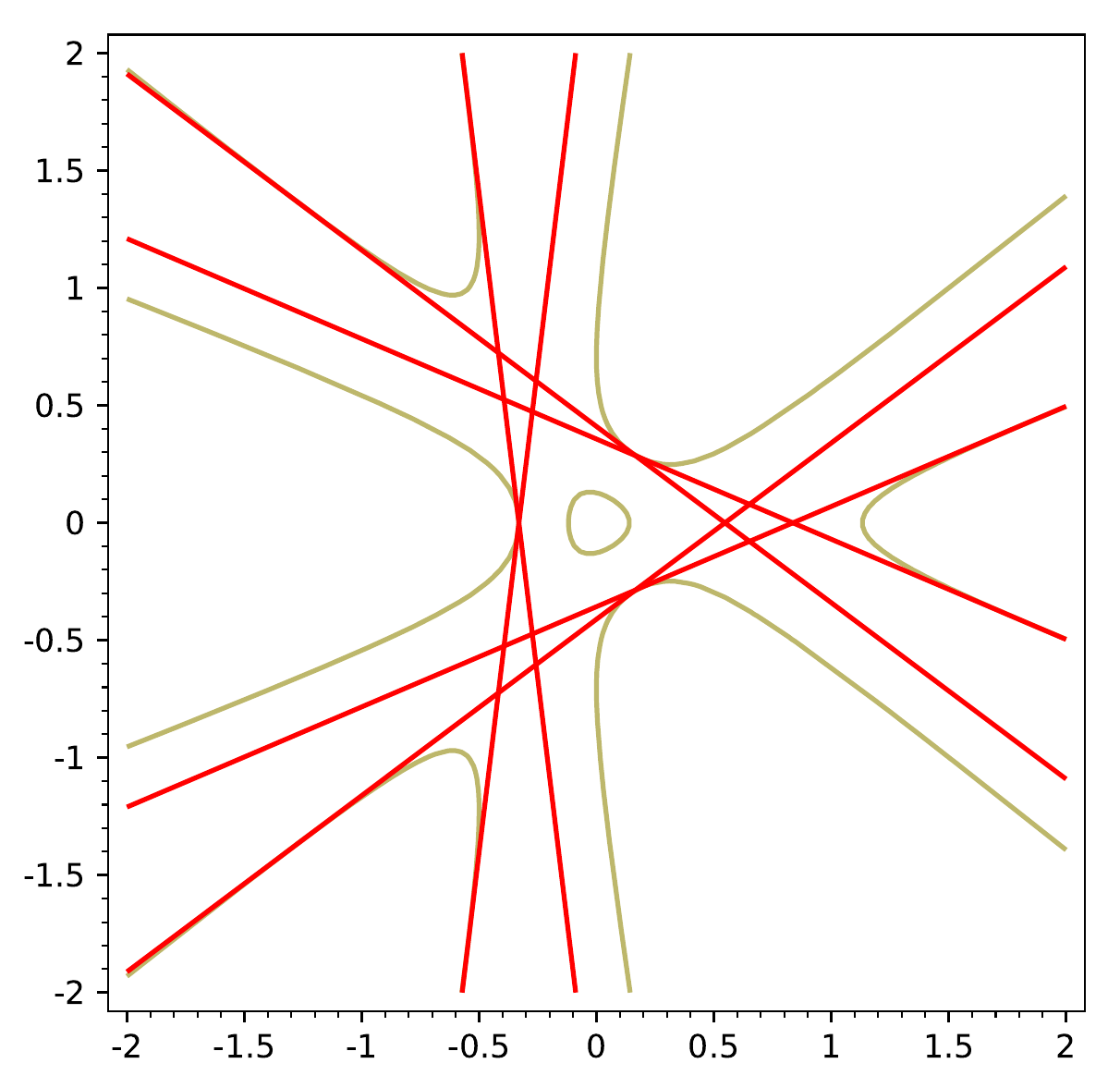}
\end{minipage}

\noindent \begin{minipage}{0.35\textwidth} \center	\noindent \scriptsize The surface $\cone A S$. \end{minipage}
\begin{minipage}{0.6\textwidth}
	\centering
	\noindent \scriptsize 27 bitangents to the curve $\cone A S \cap \tilde \Pi$.
\end{minipage}
\vspace{\baselineskip}

The code for calculating $\cone A S$, $\cone A S \cap \tilde \Pi$, and the images above can be found at the supplementary website in the file \texttt{Cone-Clebsch}.

\end{exa}

We now state the bijections. For the remainder of this subsection, $S \subset \mathbb P^3$ is a smooth cubic surface and $A$ in $\mathbb P^3$ is a point outside $S$.

\begin{lemma} \label{lemma:bijection1}
	Let $L \subset \mathbb P^3$ be a line. Then $L \subset S$ if and only if there exists $P_1, P_2$ in $L\cap S \cap S_A$ such that $T_{P_1} S = T_{P_2} S$.
\end{lemma}

Lemma~\ref{lemma:bijection1} follows from geometric considerations. Assume that $L \subset S$. Let $\Pi$ be the plane containing $L$ and $A$. By Lemma~\ref{lemma:union-three-lines} the curve $S \cap \Pi$ is singular at two points $P_1$, $P_2$ in $L$. Lemma~\ref{lemma:singular-point} implies that $\Pi = T_{P_1} S = T_{P_2} S$ and since $A \in \Pi$ we have that $P_1, P_2$ are also in $S_A$. The converse is Lemma~\ref{lemma:touch-two-points}.

\begin{lemma} \label{lemma:bijection2}
	Let $P_1, P_2$ in $S \cap S_A$ be distinct points, and $\Pi$ a plane containing~$P_1, P_2$. Then $\Pi = T_{P_1} S = T_{P_2} S$ if and only if  $\;\Pi = T_{P_1} \cone A S = T_{P_2} \cone A S$.
\end{lemma}

Lemma~\ref{lemma:bijection2} follows from a calculation.
The main geometrical idea is that $P_1, P_2$ in $S_A$ implies that $T_{P_1} S$, $T_{P_2} S$ contains the lines $\overline{P_1A}$, $\overline{P_2A}$, respectively. These lines are also contained in $\cone A S$.

Before going into the last bijection, we motivate it with the Pl\"ucker formula for calculating bitangents. One of the numbers involved in this formula is the \emph{class} of a plane curve $C$, which is the number of tangents to $C$ that contain a point $B$ not in $C$.

\begin{lemma}
	Let $C$ be a plane curve with at worse nodes and ordinary cusps as singularities. Denote by $\tau$, $\delta$, $\nu$, $\mu$ the number of bitangents, number of nodes, the degree and the class of $C$, respectively. Then
	\begin{align} \label{eq:pluecker}
	2\tau - 2 \delta = (\mu-\nu)(\mu+\nu-9).
	\end{align}
\end{lemma}

Equation~\eqref{eq:pluecker} may be obtained by eliminating variables from the equations 5.34, 5.35 and 5.36 of \cite[Chapter 5]{bcgm09}. See also Section 5.7 of \cite{fis01}.


\begin{lemma} 
	Let $\Pi \subset \mathbb P^3$ be a plane, $P$ in $S \cap S_A$ a point, and $Q$ in $\overline {PA}$. Then $\Pi$ is tangent to $\cone A S$ at $P$ if and only if $\Pi$ is tangent to $\cone A S$ at $Q$.
\end{lemma}

\begin{lemma} \label{lemma:correspondence-with-curve}
	 Let $\Pi, \tilde \Pi \subset \mathbb P^3$ be planes with $A \not \in \tilde \Pi$, $C = \cone A S \cap \tilde \Pi$ a curve, $T = \Pi \cap \tilde \Pi$ a line, and $Q_1, Q_2$ in $T$ be points. Then:
		\begin{enumerate}
			\item $T$ is tangent to $C$ at $Q_i$ if and only if $\Pi$ is tangent to $\cone A S$ at $Q_i$.
			\item $T$ is a bitangent of $C$ at points $Q_1, Q_2$ if and only if $\Pi$ is tangent to $\cone A S$ at $P_1 := \overline{AQ_1} \cap S$ and $P_2 := \overline{AQ_2} \cap S$.
			\item A point $Q \in C$ is a node if and only if $\overline{AQ}$ is tangent to $S$ at two points.
		\end{enumerate}
\end{lemma}

Now we calculate the   Pl\"ucker formulas for the curve  $C$ defined in Lemma~\ref{lemma:correspondence-with-curve}. For the nodes, by Lemma~\ref{lemma:touch-two-points} if the line $L$ is tangent to $S$ at two points, then $A \not \in L$. Thus, part 3 of  Lemma~\ref{lemma:correspondence-with-curve} implies that $\delta = 0$. 
For the degree, by Lemma~\ref{lemma:degree-cone} we get that $\nu = n(n-1) = 6$. For the class, let $B \in \tilde \Pi \setminus C$ be a point. By part 1 of Lemma~\ref{lemma:correspondence-with-curve}  and Lemma~\ref{lemma:bijection2}, a tangent line to $C$ containing $B$ corresponds to a tangent plane to $S$ containing $A$ and $B$. Recall that if a tangent of $S$ contains $A$ then the point of contact lies in the first polar $S_A$, of degree $n-1$. Likewise for $B$. Thus the point of contacts are in the set $S \cap S_A \cap S_B$, which has $n \cdot (n-1) \cdot (n-1) = 12$ points. So $\mu = 12$. Plugging values in Equation~\eqref{eq:pluecker} gives $\tau = \frac 1 2 \cdot 6 \cdot 9 = 27$. 

Apparently we have found the 27 lines on $S$, but there is no evident way of certifying that this count  corresponds to distinct lines, or if some have to be counted with a multiplicity. Thus, we only conclude Theorem~\ref{thm:one-line}, and move on to a combinatorial argument in the next section.






\subsection{Salmon's construction of 27 lines}

Fix a smooth cubic surface $S$. We sketch another argument from \cite{cay49}, a combinatorial construction that Salmon communicated to Cayley in private correspondence. Cayley wrote it down and credited Salmon in the last paragraph of \cite{cay49}. The argument proves:

\begin{thm} \label{thm:twentyseven-lines}
	If there is a line $L$ in $S$, then there are 27 distinct lines in $S$.
\end{thm}
 
Fix a line $L \subset S$. A plane $\Pi$ through $L$ cuts $S$ in a cubic curve which consists of the union of $L$ with a conic~$C$. Since $S$ is smooth, by Lemma~\ref{lemma:union-three-lines} the plane $\Pi$ intersects $S$ in three distinct lines $L, L', L''$ if and only if $C$ is singular, if and only if certain determinant of a matrix which depends on the coefficients of $C$ vanishes. This determinant gives a degree 5 polynomial. It can be shown that the roots are distinct (see Proposition 7.3 in \cite{rei88}), hence we recover  5 special planes  through~$L$. These are called \textit{ tritangent} planes through $L$ since  any such plane is equal to the tangent plane to $S$ at the intersection point of a pair of lines chosen from~$L, L', L''$.

Fix a  tritangent plane $\Pi$ through $L$. Then $S \cap \Pi$ is the union of $L$ with another two lines, say $L'$ and $L''$.  Through each line $L, L', L''$ there are other 4 distinct tritangent planes. We get 12 tritangent planes, with 2 lines in each plane that are distinct from $L$, $L'$, $L'''$. This gives 24 lines if no double counting occurs. Suppose a line is counted twice. Then it is contained in the intersection of 2 out of the 12 planes we are considering. Thus it equals either $L, L'$, or $L''$, a contradiction.
Hence, there are at least $3+12 \cdot 2  = 27$ lines in $S$. Comparing this result with the one  in Section  \ref{section:classical-way} we then get that there are exactly $27$ lines in $S$.

Moreover since there are 5 tritangent planes per line, and each plane has exactly 3 lines, one also obtains that there exist $\frac{27 \cdot 5} 3 = 45$ tritangent planes.

\subsection{\cayleys~and \textit{Triederpaare }}\label{sec:2.5}

In the course of computing a particular example, Cayley also claims in \cite{cay49}  by \emph{``merely reckoning the number of arbitrary constants''} that $S$ can be  described as the zero set of a polynomial in the form of a \cayley:
\begin{align} \label{eq:steiner-normal}
	f = l_1l_2l_3 - m_1m_2m_3,
\end{align}
where the $l_i$ and $m_j$ are linear forms.  

In this subsection we investigate the properties of a smooth surface given by an $f$ as in Equation~\eqref{eq:steiner-normal}. The aim is to establish a natural bijection between each \cayley~and a special pair of triples of tritangent planes called \textit{Triederpaar}.

We denote by $l_q$, $m_r$ both the linear forms and the planes defined by them. 
Let $\ell_{q,r} = l_q \cap m_r$ be the line of  intersection of these two planes. Note that the indices do not commute, namely $\ell_{q,r} \ne \ell_{r,q}$. The line  $\ell_{q,r}$ is in~$S$, hence  we obtain nine lines in $S$ if we can show the $\ell_{q,r}$ to be distinct. 

Suppose that $\ell_{\alpha,1} = \ell_{\beta,2}$, where $\alpha$ is not necessarily distinct from $\beta$.
%
%
	Take partial derivatives of Equation~\eqref{eq:steiner-normal}:
	
	\[
	\left(\frac{\partial l_1} {\partial x_k}l_2l_3 + l_1\frac{\partial l_2} {\partial x_k}l_3 + l_1l_2\frac{\partial l_3} {\partial x_k}\right) 
	- \left(\frac{\partial m_1} {\partial x_k}m_2m_3 +  m_1\frac{\partial m_2} {\partial x_k}m_3 + m_1m_2\frac{\partial m_1} {\partial x_k}\right)   
	\]
	Let $\gamma \in \{1,2,3\} \setminus \{\alpha, \beta \}$, and $P = \ell_{\alpha,1} \cap l_\gamma$.
	Note that $P$ is in $l_\alpha$ and $l_\gamma$, thus the first terms between the parenthesis vanish at $P$. Since  $P$ is in $m_1$ and $m_2$ the terms in the second parenthesis also vanish. Hence, $S$ is singular at $P$, contradicting that $S$ is smooth. 
	This proves that  the nine lines are distinct and the $l_i$, $m_i$ are tritangent planes of $S$ (they contain three lines each).  
	
\begin{dfn}
\label{def:tried}
Let $\mathcal L$ be a set of nine lines in $S$.
We call an unordered pair of unordered triples of tritangent planes $\{l_1, l_2, l_3\}$ and $\{m_1, m_2, m_3\}$ a \emph{Triederpaar} if each triple contains all the lines in $\mathcal L$. 
\end{dfn}

\begin{prop}\label{prop: Tried-CSeq}
	There is a natural bijection between \cayleys~for $S$ and the set of Triederpaare.
\end{prop}
\begin{proof}
By our discussion above we have that  a \cayley~induces a Triederpaar. 	

Conversely, let $\mathcal L$ be a set of nine lines in $S$ that determines a Triederpaar with tritangent planes $\{l_1, l_2, l_3\}$ and $\{m_1, m_2, m_3\}$. The linear forms corresponding to the tritangent planes are determined up to multiplication with a constant. Thus, 
	\begin{align} \label{eq:family}
		\mu l_1l_2l_3 - \nu m_1m_2m_3 = 0,
	\end{align} 
	with $[\mu:\nu] \in \mathbb P^1$, is a one parameter family $\mathcal F$ of smooth cubics  that contain  the nine lines of $\mathcal L$.
	
	Let $\mathcal F'$ be the family of smooth cubics containing $\mathcal L$. We have that $\mathcal F\subset \mathcal F'$. We now prove that $\mathcal F'$ has dimension~1 which implies  $\mathcal F=\mathcal F'$. For this we need the following claim: the line $\ell = m_1 \cap m_2$ intersects $S$ in at most three points. This claim is proven as Claim~II in the proof of Lemma~\ref{lem: claim}.
	
	Suppose $g \in \mathcal F'$. By a coordinate change we assume that $m_1$ is the plane $\{y=0\}$, and $m_2$ is $\{z=0\}$. Since  $\ell = m_1 \cap m_2$ intersects $V(g)$ in at most three points, the polynomial $g|_{y=z=0}$ is not identically zero, so we may choose a coefficient and scale it to 1. The polynomials $g|_{y=0}$, $g|_{z=0}$ are cubics in three variables, depend on ten parameters each, vanish on the triples of lines $\{\ell_{1,1},\ell_{2,1},\ell_{3,1}\}$ and $\{\ell_{1,2},\ell_{2,2},\ell_{3,2}\}$ respectively, and share the terms of $g|_{y=z=0}$, so at least one term has coefficient~1. Thus, $g|_{y=0}$ and $g|_{z=0}$ are determined, giving $10+10-4 = 16$ coefficients of $g$. This leaves 4 parameters free, namely the coefficients of $y^3$, $y^2z$, $yz^2$, $z^4$, for making $V(g)$ vanish on $C = \{\ell_{1,3},\ell_{2,3},\ell_{3,3}\}$ as well. A cubic is defined by 9 parameters, but we already have some points $C \cap \ell_{ij}$ for $j=1,2$. Specifically, the Claim~II in the proof of Lemma~\ref{lem: claim} implies that  $\ell_{1,3} \cup \ell_{2,3} \cup \ell_{3,3}$ intersects $m_1$ in three points which are contained in $V(g)$, likewise for $m_2$. Thus $V(g) \cap C$ has six points already, and we need to use another three conditions to ensure $V(G)$ vanishes on $C$, bringing the dimension of $\mathcal F'$ down to 1.
	
	
	Finally, since  $S$ is smooth it does not contain a plane so $\mu,\nu$ are nonzero. Dividing by $\mu$ we get $l_1l_2l_3 - \kappa m_1m_2m_3 = 0$. Let $P \in S$ be a point not in a line of $\mathcal L$. Since $S \cap l_i$ equals the union of three lines of $\mathcal L$, the form $l_1l_2l_3$ evaluated at $P$ is nonzero. Similarly for $m_1m_2m_3$. Thus $\kappa$ is equal to $\frac{l_1l_2l_3}{m_1m_2m_3}$ evaluated at~$P$. Including   $\kappa$ in one of the forms $m_i$ we obtain the \cayleys.
\end{proof}

\subsection{Counting  120 \cayleys~in the classical way} \label{sec:count-steiner}
In this subsection we count the number of \cayleys~by counting Triederpaare. This is done by Henderson in  Section~12 of \cite{hen11}.  There he proposes the argument below and assumes  the following lemma that we prove at the end of this subsection.
\begin{lemma}\label{lem: claim}
Let $S$ be a cubic surface and let $l_1,l_2$ be two tritangent planes containing $6$ of the $27$ lines on $S$. Then this determines a Triederpaar.
\end{lemma}

The lemma above implies that the Triederpaar is uniquely determined by a pair of tritangent planes containing $6$ of the $27$ lines of $S$.
For the first plane there are 45 possibilities. Then the number of tritangent planes passing through at least one of the three lines in the first plane is 13. This  again follows from the fact that there are $5$ tritangent planes through each line. Hence we get   32 ways to choose the second plane. 
The number of triples of tritangent planes passing through nine lines is then $\frac{45 \cdot 32 }{ 3!} = 240$. Dividing these in pairs we obtain  120 Triederpaare, and so 120 \cayleys.

\begin{proof}[Proof of Lemma~\ref{lem: claim}]
Choose a tritangent plane $l_1$.  Let $\ell_{1,1}$, $\ell_{1,2}$, $\ell_{1,3}$ be the lines in $l_1 \cap S$. Choose a tritangent plane $l_2$ not passing through the $\ell_{1,i}$. Recall Cayley's construction of the 27 lines of Section~\ref{section:classical-way}. We have seen that given a line $L \subset S$ with $L \neq\ell_{1,1}$, $\ell_{1,2}$, $\ell_{1,3}$, then there is a tritangent plane containing $L$ and exactly one $\ell_{1,i}$. Thus we can name the lines in $l_2$ as $\ell_{2,1}$, $\ell_{2,2}$, $\ell_{2,3}$ with the property that $\ell_{1,j}$ and $\ell_{2,j}$ determine a tritangent plane $m_j$.
We now set $\ell_{3,j}$ to be the line not yet named in $m_j \cap S$ and define $\mathcal{L}=\{\left(\ell_{ij}\right)_{i,j=1,2,3}\}$.
Next, we show that $\ell_{3,1}$, $\ell_{3,2}$, $\ell_{3,3}$ are coplanar, i.e. there exists a plane $l_3$, such that $\ell_{3,1},\ell_{3,2},\ell_{3,3}\subset l_3$. Then $(\{l_1,l_2,l_3\},\{m_1,m_2,m_3\})$ forms a Triederpaar, which completes the proof.
As a first step we show that $\ell_{3,1}$ intersects $\ell_{3,2}$. This follows from the following three claims:

\begin{enumerate}
\item[] Claim I: $\ell_{3,1}$ and $\ell_{3,2}$ are distinct.\\
Observe that $\ell_{3,1} = \ell_{3,2}$ if and only if the line $\ell := m_1 \cap m_2$ is contained in~$S$. Moreover if  three lines  $L$, $L'$, $L''\subset S$ intersect at a point $P$, then they are coplanar. In fact, the smoothness of $S$ implies that they are contained in the tangent plane $T_P S$. Hence as $\ell_{1,1}, \ell_{2,1} \subseteq m_1$ and $\ell_{1,2}, \ell_{2,2} \subseteq m_2$, we get $\ell_{1,1} \cap \ell_{1,2}, \ell_{2,1} \cap \ell_{2,2} \in m_1 \cap m_2 = \ell$. Thus $\ell_{1,1}, \ell_{2,1}, \ell_{1,2}, \ell_{2,2},$ are coplanar, contradicting that $m_1 \ne m_2$.
\item[] Claim II: the line $\ell = m_1 \cap m_2$ intersects $S$ in at most three points.\\
The argument for claim I also proves that $\ell \not \subset S$. Thus $\ell$ intersects $S$ in at most three points.
\item[] Claim III: the three lines in $m_i \cap S$ intersect $\ell=m_1\cap m_2$ in three different points.\\
Suppose that two lines $L$, $L' \subset m_1 \cap S$ intersect in a point $P$ of~$\ell$. By construction there is a line $L'' \subset m_2 \cap S$ intersecting one of $L$ and $L'$; say~$L$. As $L \cap L'' \in m_1 \cap m_2 = \ell$, then $L \cap L' \cap L'' = \{P\}$, and so $L, L', L''$ are coplanar. This implies that $m_1 = m_2$, contradicting that $l_{1,1} \ne l_{1,2}$. 

\end{enumerate}

Putting these claims together: let $p_1 = \ell_{1,1} \cap \ell$, $p_2 = \ell_{2,1} \cap \ell$ and $p_3 = \ell_{3,1} \cap \ell$ be the intersection points of the lines in $m_1 \cap S$ with $\ell$. Note that $p_1,p_2,p_3 \in S \cap \ell$, thus these are the three intersection points of $\ell$ with $S$ and by claim~III they are distinct. By construction $\ell_{1,1}$ intersects $\ell_{1,2}$. The intersection point is both in $m_1$ and $m_2$, thus $\ell_{1,1} \cap \ell_{1,2} = p_1$. Similarly, $\ell_{2,1} \cap \ell_{2,2} = p_2$. Hence, by Claim~III we get that $\ell_{3,2}$ intersects $\ell$ on $p_3$, proving that $\ell_{3,1}$ and $\ell_{3,2}$ intersect.

Analogously, $\ell_{3,2}$ intersects $\ell_{3,3}$ and $\ell_{3,3}$ intersects $\ell_{3,1}$, proving the lines are coplanar.
\end{proof}



\section{Computing the 120 \cayleys}\label{sec: computing}
In this section we use the isomorphism between   $S$ and the blow up of $\mathbb P^2$ at six points, denoted by  $\mathrm{Bl}_{p_1,\dots,p_6}\mathbb{P}^2$,  to first study the incidence relations between the $27$ lines and then to give 
a description of the possible Triederpaare. Then using the convention in \cite{B-K} we show the procedure  to compute the $120$ \cayleys~and we apply it to the Clebsch surface and to  the octanomial model presented in \cite{pss19}.

To begin with, we recall how to obtain the $27$ lines and the $45$ tritangent planes from $\mathrm{Bl}_{p_1,\dots,p_6}\mathbb{P}^2$. Let  $\pi:\mathrm{Bl}_{p_1,\dots,p_6}\mathbb{P}^2\to \mathbb P^2 $ be the blow-down map.  Consider the following curves:

\begin{itemize}
	\item $E_i$, the fiber over $p_i$.
	\item $F_{ij}$, the proper transform of the line $\overline{p_ip_j}$ through $p_i$, $p_j$ (that is, the curve that remains after deleting $E_i$, $E_j$ from $\inv \pi(\overline{p_ip_j})$).
	\item $G_{j}$, the proper transform of the conic $C_j$ through $\{p_1, \dots, p_6\} \setminus \{p_j\}$ (that is, the curve that remains after deleting all $E_k$ in $\inv \pi(C_j)$).
\end{itemize}

After showing that $E_i, F_{ij},$ and $G_j$ are lines (see \cite{dol12}, Section 8.3.1) we get the $27$ lines in 
$S$. 

Note that this description makes it simple to deduce the incidence relations between the lines. In fact, for any $i$ we have that $E_i$ is incident to only  10 lines. These are $F_{ij}$ and $G_j$ for $j \ne i$. The line  $G_i$ intersects $F_{ij}$ for every $j$ and $E_j$ for $j\neq i$. 
 For $F_{ij}$ observe that it is incident to $E_i$, $E_j$, $G_i$, $G_j$ and also to $F_{kl}$ with $\{i,j\} \cap \{k,l\} = \varnothing$  (the intersection point is $\inv \pi(\overline{p_ip_j} \cap \overline{p_kp_l})$). Any other intersection is excluded since it would contradict the hypothesis that $p_1, \dots, p_6$ are in general position.  
We obtain  that any line $L \subset S$ meets exactly ten other lines on $S$. Moreover, one can see that these ten lines come in pairs that are coplanar, so there are planes that intersect $S$ in 3 distinct lines $L, L', L''$. Any such   plane  is  a tritangent plane. This shows that 
a tritangent plane is determined by two types of triples of the 27 lines:
\begin{align} \label{eq:form-tritangent}
	&\text{Type I }= \{E_i, G_j, F_{ij}\}  &\text{Type II }=\{F_{ij}, F_{kl}, F_{nm}\} 
\end{align}

There are $6 \cdot 5 = 30$ possibilities for type I and   $\binom 6  2 \cdot \binom 4  2 / 3! = 15$ possibilities for type II  (this is the number of partitions $\{i,j\}, \{k,l\}, \{n,m\}$ of $\{1, \dots, 6\}$).
The description of the types of tritangent planes can be used to 
list all possible Triederpaare. Following \cite[Remark 4.6]{B-K} we may write a Triederpaar as a 3 by 3 matrix such that each entry is a line of $S$. Each row (\emph{resp.} each column) contains three lines of $S$ which determine a tritangent plane. The first triple of the Triederpaar is the one associated to the rows and the other one is associated to the columns.

Note that if two rows determine tritangent planes of type II of Equation~\eqref{eq:form-tritangent}, then the third row is also of type II. Thus in a Triederpaar we may have 0,1, or 3 rows of type~II. To summarize we have the following options:
\begin{align} \label{eq:form-steiner}
\begin{matrix}
E_i & G_j & F_{ij} \\
G_k & F_{jk} & E_{j} \\
F_{ik} & E_{k} & G_{i} 
\end{matrix}
&& 
\begin{matrix}
E_i & G_j & F_{ij} \\
G_k & E_l & F_{lk} \\
F_{ik} & F_{jl} & F_{mn} 
\end{matrix}
&&
\begin{matrix}
F_{ij} & F_{lm} & F_{kn} \\
F_{ln} & F_{ik} & F_{jm} \\
F_{km} & F_{jn} & F_{il} 
\end{matrix}
\end{align}
The Triederpaar we obtain from a matrix is invariant under permutation of rows and columns, and under transposition of the matrix. Thus, the number of essentially distinct ways of filling up the indices is $\binom 6  3 = 20$ for the first type; $\binom 6 4 \cdot \binom 4  2 = 90$ for the second type; and $\binom 5 3 = 10$ for the third type (we may assume $i=1$ and $m<n$, so one has to choose $j,k,l$ out of 5 numbers). This recovers the number 120 of Triederpaare.

In the case of the Clebsch surface  and of the octanomial model we have the explicit equations for the 27 lines in $S$. Hence   we can plug them in the matrices in~\eqref{eq:form-steiner}, compute the  associated tritangent planes and then obtain the \cayleys. 
\begin{exa}
The polynomial we gave in Example~\ref{example:clebsch} for the Clebsch surface is a coordinate transformation of the polynomial:
\[f = -{(w + x + y + z)}^{3} + w^{3} + x^{3} + y^{3} + z^{3}. \]
One \cayley{} for $f$ is:
\begin{align*}	
-3 \, {\left(w + x\right)} w x - 3 \, {\left(w + x + y\right)} {\left(w
+ x + z\right)} {\left(y + z\right)}
\end{align*}
The file \texttt{Clebsch-Tritangent-CS} in the supplementary website contains equations for the 45 tritangent planes and the 120 \cayleys{} of the Clebsch surface given by the polynomial $f$.
\end{exa}

\begin{exa}
We also perform this calculation 
 for the lines of the octanomial model of cubic surfaces presented in \cite{pss19}. 
This is a polynomial of the form $$
a\cdot xyz+b\cdot xyw+c\cdot xzw+d\cdot yzw+e\cdot x^2y+ f\cdot xy^2 +g\cdot z^2w+h\cdot zw^2
$$
where the coefficients $a,\dots,h$ are in terms of the $d_1,\dots,d_6$  as in \eqref{equ:points}.
Any smooth cubic surface in $\mathbb P^3$ can be given in this form after a coordinate change.
The calculation are done in the field $\mathbb Q(d_1,\ldots,d_6)$.
The code and the output  can be found at our supplementary website. In particular the $120$ \cayleys~are listed in the file \texttt{Output-Octanomial}.

\end{exa}

\section{\cayleys~from the anticanonical embedding}\label{sec:3}
In this section, we relate the \cayleys\; to recent work of Cueto and Deopurkar \cite{C-D}. Given a smooth cubic surface $S$, Cueto and Deopurkar construct an embedding -- which they coin the \textit{anti-canonical embedding} -- of $S$ into $\mathbb{P}^{44}$. This is given by the $45$ tritangent planes of $S$. More precisely, when considering a realization $S\subset\mathbb{P}^3$ the coordinate functions of these embeddings are (up to scaling) given by the equations of each of the tritangent planes. In \cite{C-D}, the ideal of this embedding into $\mathbb{P}^{44}$ is explicitely computed and this paves the way for several applications in the study of the tropicalization of $S$ whenever $S$ does not have Eckardt points. We note that the ideal of the anti-canonical embedding however is independent of the existence of Eckardt points.\par

We show that given a smooth cubic surface $S$ as the blow-up of $\mathbb{P}^2$ induced by the parameters $d_1,\dots,d_6$ as in  \eqref{equ:points}, we may derive the $120$ \cayleys~from the ideal of the anti-canonical embedding of $S$ without knowing the explicit equations of the $27$ lines. Our considerations  follow from the constructions in Sections $2$ and $3$ of \cite{C-D}, which is our main reference for this section. \par

As we have already observed  we can see $S$ as the blow-up of $\mathbb P^2$ at the six points in~\eqref{equ:points}. This reveals that the moduli space of marked cubic del Pezzo surfaces contains an open subset isomorphic to a dense open subset of $(\mathbb P^2)^6$ (this elaborated in detail in \cite[Section 2]{C-D}). This open subset is the projectification of the root lattice $E_6$, which is the complement of the root hyperplane arrangement. Moreover  the reflection group $W(E_6)$ acts on $(d_1,\dots,d_6)$.

Consider the symbols $x_{ij}$  and $y_{ijklmn}$. These correspond to  tritangent planes of Type I and Type II  as in Equation~\eqref{eq:form-tritangent}. 

\noindent Theorem 3.5 in \cite{C-D} states that  $S\cong \mathrm{Bl}_{p_1,\dots,p_6}\mathbb P^2$ is isomorphic to the variety in $\mathbb{P}^{44}$ cut out by the ideal $I$ in $\mathbb{K}\left[\left(x_{ij}\right)_{\substack{i,j=1,\dots,6\\i\neq j}},\left(y_{ijklmn}\right)_{\substack{\{i,j\}\sqcup\{k,l\}\sqcup\{m,n\}\\=\{1,\dots,6\}}}\right]$ generated by all $270$ $W(E_6)$-conjugates of
\begin{equation}\label{eq: linear}
(d_3-d_4)(d_1+d_3+d_4)x_{21}-(d_2-d_4)(d_1+d_2+d_4)x_{31}+(d_2-d_3)(d_1+d_2+d_3)x_{41},
\end{equation}
 and all $120$ $W(E_6)$-conjugates of
\begin{equation}
\label{equ:binom}
x_{12}x_{23}x_{31}-x_{13}x_{32}x_{21}.
\end{equation}

\noindent Now define the map $\phi:\mathbb P^3\to\mathbb P^{44}$ associated to the ring map $$\phi^*:\mathbb K[x_{ij},y_{ijklmn}]\to \mathbb K[x,y,z,w]$$ that sends $x_{ij}$ (\emph{resp.}  $y_{ijklmn}$) to the linear form defining the corresponding tritangent plane. The map $\phi$ is an embedding. In fact  the rank of the matrix associated to $\phi^*$ is $4$ since by Section \ref{sec:2.5} any four tritangent planes of a Triederpaar have empty intersection. 
Moreover, it follows immediatly from the proof of Theorem 3.5 in \cite{C-D} that the image of $\phi$ is exactly $V(I)$. 

We now  observe that pull-backs via $\phi^*$ of each of the cubics in Equation~\eqref{equ:binom} are exactly the $120$ \cayleys \;of $S$. 
In the case in which we do not have knowledge of the equations of the $27$ lines  the map $\phi$ can be computed in the following way. Let  $\mathcal L_S$ be the linear span of $ S\cong V(I)\subset\mathbb{P}^{44}$. This is defined by  the linear forms in Equation~\eqref{eq: linear} and it is isomorphic to $\mathbb{P}^3$.  The  map $\phi $ is then  the linear map
\begin{equation}
\psi :\mathbb{P}^{3}\to\mathbb{P}^{44},
\end{equation}
whose image is $\mathcal L_S$ and which is an isomorphism onto its image. Therefore, we have proved the following proposition.

\begin{prop}\label{prop:sec 4}
Let $\phi:\mathbb{P}^3\to \mathcal{L}_S\subset\mathbb{P}^{44}$ a linear map onto $\mathcal{L}_S$ and consider $S\cong\phi^{-1}(S)\subset\mathbb{P}^3$. Then, the images of the $120$ $W(E_6)$-conjugates of \eqref{equ:binom} via $\phi^*$ are the $120$ \cayleys~ of $\phi^{-1}(S)$ .
\end{prop}


\noindent The explicit linear forms defining $\psi$ can be at the supplementary website in the file \texttt{Code-AnticanonicalEmbedding}.
The  pullbacks of the  $120$ cubic binomials  are included in  \texttt{Output-AnticanonicalEmbedding}.

\section*{Acknowledgements}
We are grateful to Angie Cueto, Marta Panizzut, Emre Sert\"oz, and Bernd Sturmfels for many helpful discussions and suggestions regarding this project. Furthermore, we are indebted to Angie Cueto for her support regarding software related to \citep{C-D}. We would like to thank Dominic Bunnett, Jan Draisma, Emre Sert\"oz and the anonymous referees for their feedback on drafts of this article. It improved the exposition and enriched the content. M.~A.~Hahn and S.~Lamboglia gratefully acknowledge support of the LOEWE research unit Uniformized Structures in Arithmetic and Geometry.

\bibliography{C-S-arxiv}{}
\end{document}